\documentclass[sn-mathphys,Numbered]{sn-jnl}

\usepackage{graphicx}%
\usepackage{multirow}%
\usepackage{amsmath,amssymb,amsfonts}%
\usepackage{amsthm}%
\usepackage{mathrsfs}%
\usepackage[title]{appendix}%
\usepackage{xcolor}%
\usepackage{textcomp}%
\usepackage{manyfoot}%
\usepackage{booktabs}%
\usepackage{algorithm}%
\usepackage{algorithmicx}%
\usepackage{algpseudocode}%
\usepackage{listings}%
\usepackage{placeins}
\numberwithin{equation}{section}

\theoremstyle{thmstyleone}%
\newtheorem{theorem}{Theorem}
\theoremstyle{thmstyletwo}%
\newtheorem{remark}{Remark}%
\newtheorem{lemma}{Lemma}%
\newtheorem{corollary}{Corollary}

\theoremstyle{thmstylethree}%

\begin{document}

\title[Article Title]{Maximum principle for discrete-time control systems driven by fractional noises and related backward stochastic difference equations}

\author[1]{\fnm{Yuecai} \sur{Han}}\email{hanyc@jlu.edu.cn}

\author*[1]{\fnm{Yuhang} \sur{Li}}\email{yuhangl22@mails.jlu.edu.cn}

\equalcont{These authors contributed equally to this work.}

\affil[1]{\orgdiv{School of Mathematics}, \orgname{Jilin University}, \orgaddress{ \city{Changchun}, \postcode{130012},  \state{Jilin Province}, \country{China}}}



\abstract{In this paper, the optimal control for discrete-time systems driven by fractional noises is studied. A stochastic maximum principle is obtained by introducing a backward stochastic difference equation contains both fractional noises and the constructed white noises. The solution of the backward stochastic difference equations is also investigated. As an application, the linear quadratic case is considered to illustrate the main results.}




\maketitle

\section{Introduction}
From last century,  stochastic maximum principle (SMP) and backward stochastic differential equations (BSDEs) are studied popularly to deal with optimal control problems of stochastic systems. For SMP,  some original works refer to \cite{kushner1972necessary,bismut1978introductory,bensoussan2006lectures}. Right up to 1990, Peng \cite{peng1990general} obtain the stochastic maximum principle for a general control system, where the control domain could be non-convex and the diffusion term contains control process. Then SMP for many different control problems are investigated, such as near-optimal control \cite{zhou1998stochastic}, doubly stochastic control systems \cite{han2010maximum,zhang2011maximum}, mean-field optimal control \cite{li2012stochastic,yong2013linear} and delayed control problems \cite{chen2010maximum,yu2012stochastic}. For BSDEs, as the adjoint equation in SMP, it is originated from Bismut \cite{bismut1973conjugate}. In 1990, Pardoux and Peng \cite{pardoux1990adapted} prove the existence and uniqueness of the solution of nonlinear BSDEs. Then some related early works refer to \cite{peng1993backward,el1997backward}.

For SMP for discrete-time systems, the original work is by Lin and Zhang \cite{lin2014maximum}, a type of backward stochastic difference equations (BS$\Delta$Es) is introduced as adjoint equation.
Based on this work, much progress has been made by
researchers. Discrete-time stochastic games are studied by Wu and Zhang \cite{wu2022maximum}. Dong et al. obtain the SMP for discrete-time systems with mean-field \cite{dong2022maximum} and delay \cite{dong2023maximum}. Ji and Zhang investigate the infinite horizon recursive optimal control and infinite horizon BS$\Delta$E. But as we known, there are poorly works consider the discrete-time control problems with fractional noises or other ``colored noises".

In this paper, motivated by stochastic optimal control of continuous systems driven by fractional Brownian motion (FBM), we concern the discrete-time optimal control for systems driven by fractional noises.
Let $H\in(0,\,1)$ be a fixed constant, which is called Hurst parameter. The $m$-dimensional FBM $B_t^H=\left(B_1^H(t),\cdots,B_m^H(t)\right), t\in[0,T]$ is a continuous, mean 0 Gaussian process with the covariance
\begin{align*}
\mathbb{E}[B_i^H(t)B_j^H(s)]=\frac{1}{2}\delta_{ij}(t^{2H}+s^{2H}-| t-s|^{2H}),
\end{align*}
$i,\,j=1,\dots,m$. Moreover, the FBM could be generated by a standard Brownian motion through
\begin{align*}
B_j^H(t)=\int_0^t Z_H(t,\,s) dB_j(s), \quad 1\le j\le m,
\end{align*}
for some explicit functions $Z_H(\cdot,\cdot)$.
For control problem of continuous systems driven by FBM, Han et al. \cite{han2013maximum} firstly obtain the maximum principle for general systems. Some other related works refer to \cite{hu2005stochastic,hu2009backward,sun2021stochastic}.

For our problem, the state equation is
\begin{align*}
\left\{\begin{array}{ll}
X_{n+1}=X_n+b(n,X_n,u_n)+\sigma(n,X_n,u_n)\xi_n^H,\, 0\le n\le N-1,
\\X_0=x,
\end{array}\right.
\end{align*}
to minimize the cost function
\begin{align*}
J(u)=\mathbb{E}\left[\sum_{n=0}^{N-1}l(n,X_n,u_n)+\Phi(X_N)\right].
\end{align*}
Here $\xi^H$ is called a fractional noise describe by the increment of a FBM.

To obtain the stochastic maximum principle, the following BS$\Delta$E is introduced as adjoint equation:
\begin{align}\label{adjoint}
\left\{\begin{array}{ll}
p_n+q_n\eta_n=p_{n+1}+b_x^*(n+1)p_{n+1}+b(n+1,n+1)\sigma_x^*(n+1)q_{n+1}\\
\qquad\qquad\qquad+l_x^*(n+1)+\sigma_x^*(n+1)p_{n+1}\xi_{n+1}^H,
\\p_N=\Phi_x(X^*_N),\\
q_N=0,
\end{array}\right.
\end{align}
where $\eta$ is a Gaussian white noises constructed by $\xi^H$.

\begin{remark}
Our techniques can also be used in control problem with some other noises $\omega$, such as AR(p) or MA(q) model. Indeed, we just need the following two conditions are equivalent:

(i):$\mathbb{E}\left[\left(\sum_{k=0}^na_1(n,k)\omega_k\right)\cdot\left(\sum_{l=0}^ma_2(m,l)\omega_l\right)\right]=0$.

(ii):$\sum_{k=0}^na_1(n,k)\omega_k$ and $\sum_{l=0}^ma_2(m,l)\omega_l$ are independent.

\noindent For all determined functions $a_1,a_2$ and $m,n\in \mathbb{Z}^+$.
\end{remark}

A difficulty is that it is hard to estimate the $L^2_\mathcal{F}$-norm for $X,p,q$ and the variation equation appears in section 3, because of the dependence of $\xi^H$ and its coefficient. We deal with this with more complex calculations, then we prove the uniqueness of the state S$\Delta$E and adjoint BS$\Delta$E, and obtain the maximum principle. 

The rest of this paper is organized as follows. In section 2, we introduce the BS$\Delta$E driven by both fractional noise and white noise, and prove the existence and the uniqueness of the solution of this type of BS$\Delta$E. 
In section 3, we obtain the stochastic maximum principle by proving  the existence and uniqueness of the state S$\Delta$E and showing the convergence of the variation equation. 
In section 4, the linear quadratic case is investigated to illustrate the main results.

\section{Backward stochastic difference equations}
Let $(\Omega,\mathcal{F},\{\mathcal{F}_n\}_{n\in\mathbb
{Z}^+},P)$ be a filtered probability space, $\mathcal{F}_0\subset \mathcal{F}$ be a sub $\sigma$-algebra. $\{\xi_n^H\}_{n\in\mathbb
{Z}^+}$ is a sequence of fractional noises described by the increment of a $m$-dimensional fractional Brownian motion, namely, $\xi_n^H=B^H(n+1)-B^H(n)$. Define the filtration $\mathbb{F}=(\mathcal{F}_n)_{0\le n\le N}$ by $\mathcal{F}_n=\mathcal{F}_0\lor\sigma(\xi^H_0,\xi_1^H,...,\xi_{n-1}^H)$.

Denote by $L^\beta(\mathcal{F}_n;\mathbb{R}^n)$, or $L^\beta(\mathcal{F}_n)$ for simplify,  the set of all $\mathcal{F}_n$-measurable  random variables $X$ taking values in $\mathbb{R}^n$, such that $\mathbb{E}\|X\|^\beta<+\infty$. Denote by $L^\beta_\mathcal{F}(0,T;\mathbb{R}^n)$, or $L^\beta_\mathcal{F}(0,T)$ for simplify, the set of all $\mathcal{F}_n$-adapted process $X=(X_n)_{n\in\mathbb
{Z}^+}$ such that 
\begin{align*}
\|X_\cdot\|_{\beta}=\left(\sum_{n=0}^N\mathbb{E}\|X_n\|^\beta\right)^{\frac{1}{\beta}}<+\infty.
\end{align*}
Then we consider the following BS$\Delta$E ,
\begin{align}\label{BSDE}
\left\{\begin{array}{ll}
Y_n+Z_n\eta_n=Y_{n+1}+f(t,Y_{n+1},Z_{n+1})+g(t,Y_{n+1},Z_{n+1})\xi_{n+1}^H,
\\Y_N=y,
\end{array}\right.
\end{align}
where $y\in L^{2a}(\mathcal{F}_N)$ for some $a>1$, and $\eta_n\in\mathcal{F}_{n+1}$, which will be defined in the following lemma, is a sequence of independent Gaussian random variables generated by $\{\xi_n^H\}$. We also assume the following conditions for $f, g$:
\begin{enumerate}
	\item[(H2.1)]  $f$ and $g$ are adapted processes: $f(\cdot,y,z), g(\cdot,y,z)\in L^{2a}_{\mathcal{F}}(1,N)$ for all $y\in \mathbb{R}^n, z\in\mathbb{R}^{n\times m}$.

\item[(H2.2)] There are some constants $L>0$, such that 
\begin{align*}
|f(n,y_1,z_1)-f(n,y_2,z_2)|&+|g(n,y_1,z_1)-g(n,y_2,z_2)|\le L\left(|y_1-y_2|+|z_1-z_2|\right),\\
&\forall n\in[1,N],\quad y_1,y_2\in\mathbb{R}^n,\quad z_1,z_2\in\mathbb{R}^{n\times m}.
\end{align*}
\item[(H2.3)] There exists a $\mathcal{F}_N$-measurable functions $ f_1, g_1$, such that $f(N,y,z)=f_1(y)$ and $ g(N,y,z)=g_1(y)$ for all $y,z$.
\end{enumerate}

Then we show the construction and properties of $\{\eta_n\}$.

\begin{lemma}\label{lem1}
Let $a(\cdot,\cdot)$ is given by $a(i-1,j-1)=(B^{-1})_{ij}$, where $B=(b_{ij})$ is given by
\begin{align}\label{bij}
b_{ij}=\mathbf{1}_{\{j\le i\}}\frac{\rho(i,j)-\sum_{k=0}^{j-2}b_{i,k+1}b_{j,k+1}}{b_{jj}}.
\end{align}
Then $\eta_n=\sum_{k=0}^na(n,k)\xi_m^H$ is a sequence of independent Gaussian random variables such that $\eta_n$ is $\mathcal{F}_{n+1}$-measurable but independent with $\mathcal{F}_n$.
\end{lemma}
\begin{proof}
From the properties of fractional Brownian motion, we know
\begin{align*}
\left(\xi_0^H,\xi_1^H,...,\xi_{N-1}^H\right)^T\sim \mathcal{N}\left(0_{N\times1},\Sigma_{N\times N}\right),
\end{align*}
where $\Sigma_{ij}=\rho(i-1,j-1)$. So that there exists invertible matrix $A$ such that 
\begin{align}\label{xieta}
\left(\eta_0^H,\eta_1^H,...,\eta_{N-1}^H\right)^T=A\left(\xi_0^H,\xi_1^H,...,\xi_{N-1}^H\right)^T\sim \mathcal{N}\left(0_{N\times1},I_{N\times N}\right).
\end{align}
Notice that $A$ is not unique, we choose the lower triangular one since $\eta_n$ is $\mathcal{F}_n$-measurable. But it is hard to obtain $A$ directly, we determine $B=A^{-1}$ at first by checking $b_{ij}$ satisfy equality (\ref{bij}). 

Let $b(i-1,j-1)=b_{ij}$,  rewrite (\ref{xieta}) as 
\begin{align*}
B\left(\eta_0^H,\eta_1^H,...,\eta_{N-1}^H\right)^T=\left(\xi_0^H,\xi_1^H,...,\xi_{N-1}^H\right)^T,
\end{align*}
or
\begin{align*}
\xi_n^H=\sum_{j=1}^{N}b_{n+1,j}\eta_j=\sum_{k=0}^{N-1}b(n,k)\eta_k.
\end{align*}
Since $\eta_k$ is $\mathcal{F}_{k+1}$-measurable and independent with $\mathcal{F}_k$, $b(n,k)=0 $ for $ k>n$.
Then let $\mathbb{E}\left[\xi_n\xi_m\right]=\sum_{k=0}^mb(n,k)b(m,k)=\rho(n,m)$ for $m\le n$, we obtain the induction formula for $b(\cdot,\cdot)$:
\begin{align*}
b(n,m)=\frac{\rho(n,m)-\sum_{k=0}^{m-1}b(n,k)b(m,k)}{b(m,m)},
\end{align*}
which equally to (\ref{bij}).
\end{proof}
\begin{remark}
It is clear that 
$\{\xi_n\}$ and $\{\eta_n\}$ generate the same filtration, namely,  $\sigma(\xi_0,\xi_1,...,\xi_{n})=\sigma(\eta_0,\eta_1,...,\eta_{n})$ for $0\le n\le N-1$.
\end{remark}

Then we show the solvability of BS$\Delta$E (\ref{BSDE}).
\begin{theorem}
Assume that assumptions (H2.1)-(H2.3) hold. Then BS$\Delta$E (\ref{BSDE}) has a unique solution.
\end{theorem}

\begin{proof}
We first define $Y_{N-1}$ by constructing
\begin{align*}
M_n=\mathbb{E}\left[y+f_1(y)+g_1(y)\xi_{N}^H|\mathcal{F}_n\right].
\end{align*}
Through the Lipschitz's condition, $
\text {Hölder}
$'s inequality and the fact $y\in L(\mathcal{F}_N)$, we have
\begin{align}\label{2.4}
\mathbb{E}|M_n|^{2a\delta}&\le \mathbb{E}\left|y+f_1(y)+g_1(y)\xi_{N}^H\right|^{2a\delta}\notag\\
&\le C\left[\mathbb{E}|y|^{2a\delta}+\mathbb{E}|f(N,0,0)|^{2a\delta}\right]\notag\\
&\quad+C\left[\left(\mathbb{E}|y|^{2a}\right)^\delta+\left(\mathbb{E}|g(N,0,0)|^{2a}\right)^\delta\right]\left(\mathbb{E}|\xi_{N}|^{2a\delta\frac{1}{1-\delta}}\right)^{1-\delta}\notag\\
&<+\infty,
\end{align}
for $\delta\in(\frac{1}{a},1)$. Then it is a straightforward result that $\mathbb{E}|M_n|^2\le \left(\mathbb{E}|M_n|^{2a\delta}\right)^{\frac{1}{a\delta}}<+\infty$, so that $M_n$ is a square integrable martingale. Similar to formula (2.5) of \cite{}, there exists a unique adapted process $Z$, such that
\begin{align*}
M_n=M_0+\sum_{k=0}^{n-1}Z_k\eta_k,\quad  0\le n\le N.
\end{align*}
Rewrite the above equality as 
\begin{align*}
M_{N-1}+Z_{N-1}\eta_{N-1}=y+f_1(y)+g_1(y)\xi_{N}^H,
\end{align*}
which shows $Y_{N-1}=M_{N-1}$. Multiply $\eta_{N-1}$ on both side and take $\mathbb{E}[\cdot|\mathcal{F}_{N-1}]$, we obtain $Z_{N-1}$ by
\begin{align*}
Z_{N-1}=\mathbb{E}\left[\eta_{N-1}[y+f_1(y)+g_1(y)\xi_{N}^H]|\mathcal{F}_{N-1}\right].
\end{align*}
Similar to formula (\ref{2.4}), we have $\mathbb{E}|Z_{N-1}|^{2a\delta}<+\infty$. Up to now, we obtain the unique pair $(Y_{N-1},Z_{N-1})\in L^{2a\delta}(\mathcal{F}_{T-1})\times L^{2a\delta}(\mathcal{F}_{T-1})$. 

Then by induction, for all $0\le n\le N-2$, if  $(Y_{n+1},Z_{n+1})\in L^{2b}(\mathcal{F}_{n+1})\times L^{2b}(\mathcal{F}_{n+1})$ for some $b>1$. Define
\begin{align*}
\tilde{M}_k=\mathbb{E}\left[Y_{n+1}+f(n+1,Y_{n+1},Z_{n+1})+g(n+1,Y_{n+1},Z_{n+1})\xi_{n+1}^H|\mathcal{F}_k\right]
\end{align*}
for $0\le k\le n+1$. Then
we show $\tilde{M}_k\in L^{2b\delta}(\mathcal{F}_k)$:
\begin{align*}
\mathbb{E}|\tilde{M}_k|^{2b\delta}&=\mathbb{E}\left|Y_{n+1}+f(n+1,Y_{n+1},Z_{n+1})+g(n+1,Y_{n+1},Z_{n+1})\xi_{n+1}^H\right|^{2b\delta}\notag\\
&\le C\left[\mathbb{E}|Y_{n+1}|^{2b\delta}+\mathbb{E}|Z_{n+1}|^{2b\delta}+\mathbb{E}|f(n+1,0,0)|^{2b\delta}\right]\notag\\
&\quad+C\left[\mathbb{E}|Y_{n+1}|^{2b\delta}+\mathbb{E}|Z_{n+1}|^{2b\delta}+\mathbb{E}|g(n+1,0,0)|^{2b\delta}\right]\left(\mathbb{E}|\xi_{N-1}|^{2b\delta\frac{1}{1-\delta}}\right)^{1-\delta}\notag\\
&<+\infty.
\end{align*}
Let $Y_n=\tilde{M}_n$ and
\begin{align*}
Z_n=\mathbb{E}\left[\eta_n\left[Y_{n+1}+f(n+1,Y_{n+1},Z{n+1})+g(n+1,Y_{n+1},Z_{n+1})\xi_{n+1}^H\right]|\mathcal{F}_n\right],
\end{align*}
which implies
\begin{align*}
Y_n+Z_n\eta_n=Y_{n+1}+f(t,Y_{n+1},Z_{n+1})+g(t,Y_{n+1},Z_{n+1})\xi_{n+1}^H.
\end{align*}
Choosing $\delta\in[a^{-1/N},1)$, we obtain the unique process pair $(Y,Z)\in L^2_{\mathcal{F}}(0,N)\times L^2_{\mathcal{F}}(0,N-1)$.
\end{proof}

\section{A maximum principle}
In this section, we study the optimal control problem for discrete-time systems driven by fractional noises. Let $(\Omega,\mathcal{F},\{\mathcal{F}_n\}_{n\in\mathbb
{Z}^+},P)$ be a filtered probability space, $\mathcal{F}_0\subset \mathcal{F}$ be a sub $\sigma$-algebra and $\mathbb{F}=(\mathcal{F}_n)_{0\le n\le N}$ be the filtration defined by $\mathcal{F}_n=\mathcal{F}_0\lor\sigma(\xi^H_0,\xi_1^H,...,\xi_{n-1}^H)$. The state equation is
\begin{align}\label{state}
\left\{\begin{array}{ll}
X_{n+1}=X_n+b(n,X_n,u_n)+\sigma(n,X_n,u_n)\xi_n^H,\, 0\le n\le N-1,
\\X_0=x,
\end{array}\right.
\end{align}
with the cost function
\begin{align}\label{cost}
J(u)=\mathbb{E}\left[\sum_{n=0}^{N-1}l(n,X_n,u_n)+\Phi(X_N)\right].
\end{align}
Here $x\in L^{2a}(\mathcal{F}_0)$ is independent with $\{\xi_n\}$, 
$b(n,x,u)$ and $\sigma(n,x,u) $ are measurable functions on $[0,N-1]\times \mathbf{R}^d\times \mathbf{R}^k$ with values in $\mathbf{R}^d$ and $\mathbf{R}^{d\times m}$, respectively.
$l(n,x,u)$ and $\Phi(x) $ be measurable functions on $[0,N-1]\times \mathbf{R}^d\times \mathbf{R}^k$ and $\mathbf{R}^d$, respectively, with values in $\mathbf{R}$.

Denote by $\mathbb{U}$ the set of progressively measurable process
\textbf{u}$=(u_n)_{0\le n\le N-1}$ taking values in a given closed-convex set $\textbf{U}\subset \mathbb{R}^k$ and satisfying $\mathbb{E}\sum_{n=0}^{N-1}|u_n|^{2a} <+\infty$. The problem is to find an optimal control $u^*\in\mathbb{U}$ to minimized the cost function, i.e.,
\begin{align*}
J(u^*)=\inf_{u\in\mathbb{U}}J(u).
\end{align*}

To simplify the notation without losing the generality, we assume that $d=k=m=1$. We give the following assumptions:
\begin{enumerate}
	\item[(H3.1)]  $b$ and $\sigma$ are adapted processes: $b(\cdot,x,u), \sigma(\cdot,y,z)\in L^{2a}_{\mathcal{F}}(0,N-1)$ for all $x,u\in\mathbb{R}$.

\item[(H3.2)] $b(n,x,u),\sigma(n,x,u)$ are differential w.r.t $(x,u)$ and there exists some constants $L>0$, such that 
\begin{align*}
|
\phi(n,x_1,u_1)-\
&\phi(n,x_2,u_2)|\le L\left(|x_1-x_2|+|u_1-u_2|\right),\\
&\forall n\in[0,N-1],\quad\forall x_1,x_2,u_1,u_2\in\mathbb{R},
\end{align*}
for $\phi=b,\sigma,b_x,\sigma_x,b_u,\sigma_u.$
\item[(H3.3)] $l(n,x,u),\Phi(x)$ are differential w.r.t $(x,u)$ and 
\begin{align*}
|l_x(n,x,u)|+|\Phi_x(x)|\le L(1+|x|+|u|),\quad \forall n\in[0,N-1], \, x,u\in\mathbb{R}.
\end{align*}
\item[(H3.4)] $b(N,x,u)=\sigma(N,x,u)=l(N,x,u)=0,$ for all $(x,u)$.
\end{enumerate}
Then we show the solvability of the state equation.
\begin{lemma}\label{lem2}
Assume that assumptions (H3.1)-(H3.4) hold and $u\in\mathbb{U}$, then S$\Delta$E (\ref{state}) has a unique solution $X\in L_{\mathcal{F}}^2(0,N)$.
\end{lemma}
\begin{proof}
Since $x\in L^{2a}(\mathcal{F}_0)$, by assumptions (3.1) and (3.2),  we have
\begin{align*}
\mathbb{E}|X_1|^{2a\delta}&\le C\mathbb{E}\left[|x|^{2a\delta}+|f(0,0,0)|^{2a\delta}+|g(0,0,0)|^{2a\delta}|\xi_0^H|^{2a\delta}+|x|^{2a\delta}|\xi_0^H|^{2a\delta}\right]\\
&\le C\left[\left(\mathbb{E}|x|^{2a}\right)^\delta+\left(\mathbb{E}|f(0,0,0)|^{2a}\right)^\delta\right] \\
&\quad +C\left[\left(\mathbb{E}|x|^{2a}\right)^\delta+\left(\mathbb{E}|g(0,0,0)|^{2a}\right)^\delta\right]\times\left(\mathbb{E}|\xi_0^H|^{\frac{2a\delta}{1-\delta}}\right)^{1-\delta}\\
&<+\infty.
\end{align*}
for some $C>0$ and $\delta\in[a^{-1/N},1)$.

Then, by induction, if $X_n\in L^{2a\delta^n}(\mathcal{F}_n)$, we show $X_n\in L^{2a\delta^{n+1}}(\mathcal{F}_{n+1})\subset L^{2a}(\mathcal{F}_{n+1}), n=0,1,...,N-1$:
\begin{align*}
\mathbb{E}|X_{n+1}|^{2a\delta^{n+1}}&\le C\mathbb{E}\left[|X_n|^{2a\delta^{n+1}}+|f(n,0,0)|^{2a\delta^{n+1}}+|g(n,0,0)|^{2a\delta^{n+1}}|\xi_0^H|^{2a\delta^{n+1}}\right]\\
&\quad+C\mathbb{E}\left[|X_n|^{2a\delta^{n+1}}|\xi_0^H|^{2a\delta^{n+1}}\right]\\
&\le C\mathbb{E}\left[|X_n|^{2a\delta^{n+1}}+|f(0,0,0)|^{2a\delta^{n+1}}\right] \\
&\quad +C\left[\left(\mathbb{E}|x|^{2a\delta^n}\right)^\delta+\left(\mathbb{E}|g(0,0,0)|^{2a\delta^n}\right)^\delta\right]\times\left(\mathbb{E}|\xi_0^H|^{\frac{2a\delta^{n+1}}{1-\delta}}\right)^{1-\delta}\\
&<+\infty.
\end{align*}
Thus, the solution of (\ref{state}) in $ L^2_{\mathcal{F}}(0,N)$.

Uniqueness. Let $\tilde{X}$ and $X$ be two solutions of S$\Delta$E (\ref{state}), so that $\mathbb{E}|\tilde{X}_0-X_0|^{2a}=0$. If $\mathbb{E}|\tilde{X}_n-X_n|^{2a\delta^n}=0$, we have
\begin{align*}
\mathbb{E}|\tilde{X}_{n+1}-X_{n+1}|^{2a\delta^{n+1}}\le& C\mathbb{E}\left[|\tilde{X}_n-X_n|^{2a\delta^{n+1}}+|\tilde{X}_n-X_n|^{2a\delta^{n+1}}|\xi_n^H|^{2a\delta^{n+1}}\right]\\
\le& C\mathbb{E}|\tilde{X}_n-X_n|^{2a\delta^{n+1}}+C\left(\mathbb{E}|\tilde{X}_n-X_n|^{2a\delta^n}\right)^\delta\times\left(\mathbb{E}|\xi_n^H|^{\frac{2a\delta^{n+1}}{1-\delta}}\right)^{1-\delta}\\
=& 0,
\end{align*}
for $n=0,1,...,N-1$. Then we conclude $\sum_{n=0}^N\mathbb{E}|\tilde{X}_n-X_n|^2=0$, which shows the solution is unique.
\end{proof}
\begin{remark}\label{rem2}
Indeed, we could obtain the unique solution of 
S$\Delta$E (\ref{state}) in $L^{\beta}_{\mathcal{F}}(0,N)$ for $\beta\in(2,2a)$ as long as we take $\delta\in\left(\sqrt[N]{\frac{\beta}{2a}},1\right).$
\end{remark}
For any $u\in\mathbb{U}$ and $\varepsilon\in(0,1)$, let
\begin{align*}
u_n^\varepsilon=(1-\varepsilon)u^*_n+\varepsilon u_n:=u^*_n+\varepsilon v_n.
\end{align*}
Denote
\begin{align*}
\phi^*(n)=\phi(n,X^*_n,u^*_n),\\
\phi^\varepsilon(n)=\phi(n,X^\varepsilon_n,u^\varepsilon_n)
\end{align*}
for $\phi=b,\sigma,l,b_x,\sigma_x,l_x,b_u,\sigma_u,l_u$.

Define the variation equation by 
\begin{align*}
\left\{\begin{array}{ll}
V_{n+1}=V_n+b_x^*(n)V_n+b_u^*v_n+\left[\sigma_x^*(n)V_n+\sigma_u^*(n)v_n\right]\xi_n^H,\, 0\le n\le N-1,
\\V_0=0.
\end{array}\right.
\end{align*}
Then we have the following convergence result.
\begin{lemma}\label{lem3}
Let $X^\varepsilon$, $X^*$ be the corresponding state equation to $u^\varepsilon$, $u^*$, respectively. Then we have 
\begin{align}\label{3.3}
\sum_{n=0}^N\mathbb{E}|X_n^\varepsilon-X^*_n|^2\le O(\varepsilon^2),
\end{align}
and
\begin{align}\label{3.4}
\lim_{\varepsilon\to 0}\sum_{n=0}^N\mathbb{E}\left|\frac{X^\varepsilon_n-X^*_n}{\varepsilon
}-V_n\right|^2=0.
\end{align}
\end{lemma}
\begin{proof}
The proof of (\ref{3.3}) follows lemma \ref{lem2}, if $\mathbb{E}|X_n^\varepsilon-X_n^*|^{2a\delta^n}=O\left(\varepsilon^{2a\delta^n}\right)$, then 
\begin{align*}
\mathbb{E}|X^\varepsilon_{n+1}-X^*_{n+1}|^{2a\delta^{n+1}}\le& C\mathbb{E}\left[|X^\varepsilon_n-X^*_n|^{2a\delta^{n+1}}+|X^\varepsilon_n-X^*_n|^{2a\delta^{n+1}}|\xi_n^H|^{2a\delta^{n+1}}\right]\\
\le& C\left(\mathbb{E}|X^\varepsilon_n-X^*_n|^{2a\delta^n}\right)^\delta+C\left(\mathbb{E}|X^\varepsilon_n-X^*_n|^{2a\delta^n}\right)^\delta\times\left(\mathbb{E}|\xi_n^H|^{\frac{2a\delta^{n+1}}{1-\delta}}\right)^{1-\delta}\\
=& O\left(\varepsilon^{2a\delta^{n+1}}\right).
\end{align*}
So that
\begin{align*}
\sum_{n=0}^N\mathbb{E}|X^\varepsilon_n-X^*_n|^2\le \sum_{n=0}^N\left(\mathbb{E}|X^\varepsilon_n-X^*_n|^{2a\delta^n}\right)^{\frac{1}{a\delta^n}}=O(\varepsilon^2).
\end{align*}

Denote $\hat{X}^\varepsilon=\frac{X^\varepsilon-X^*}{\varepsilon}-V$, it follows
\begin{align*}
\hat{X}^\varepsilon_{n+1}=&\hat{X}^\varepsilon_{n}+\frac{b^\varepsilon(n)-b^*(n)}{\varepsilon}-b_x^*(n)V_n-b_u^*(n)v_n\\
&+\left[\frac{\sigma^\varepsilon(n)-\sigma^*(n)}{\varepsilon}-\sigma_x^*(n)V_n-\sigma_u^*(n)v_n\right]\xi_n^H\\
=&[1+b_x^*(n)]\hat{X^\varepsilon_n}+\frac{\tilde{b}_x^\varepsilon(n)-b^*_x(n)}{\varepsilon}[X^\varepsilon_n-X^*_n]+[\tilde{b}_u^\varepsilon(n)-b^*_u(n)]v_n\\
&+\left[\sigma_x^*(n)\hat{X^\varepsilon_n}+\frac{\tilde{\sigma}_x^\varepsilon(n)-\sigma^*_x(n)}{\varepsilon}[X^\varepsilon_n-X^*_n]+[\tilde{\sigma}_u^\varepsilon(n)-\sigma^*_u(n)]v_n\right]\xi^H_n,
\end{align*}
where
\begin{align*}
\tilde{\phi}^\varepsilon(n)=\int_0^1 \phi(n,X^*_n+\theta(X^\varepsilon_n-X^*_n),u^*_n+\theta\varepsilon v_n)d\theta,
\end{align*}
for $\phi=b_x,b_u,\sigma_x,\sigma_u$. Then
\begin{align*}
\mathbb{E}\left|\hat{X}^\varepsilon_{n+1}\right|^{2a\delta^{n+1}}\le& C\mathbb{E}\left|\hat{X}^\varepsilon_{n}\right|^{2a\delta^{n+1}}+C\mathbb{E}\|\tilde{b}_u^\varepsilon(n)-b^*_u(n)\|^{2a\delta^{n+1}}|v_n|^{2a\delta^{n+1}}\\
&+C\mathbb{E}\|\tilde{b}_x^\varepsilon(n)-b^*_x(n)\|^{2a\delta^{n+1}}\left|\varepsilon^{-1}(X^\varepsilon_n-X^*_n)\right|^{2a\delta^{n+1}}\\
&+C\mathbb{E}\Bigg[\left|\hat{X}^\varepsilon_{n}\right|^{2a\delta^{n+1}}+\|\tilde{\sigma}_u^\varepsilon(n)-\sigma^*_u(n)\|^{2a\delta^{n+1}}|v_n|^{2a\delta^{n+1}}\\
&\qquad+\|\tilde{\sigma}_x^\varepsilon(n)-\sigma^*_x(n)\|^{2a\delta^{n+1}}\left|\varepsilon^{-1}(X^\varepsilon_n-X^*_n)\right|^{2a\delta^{n+1}}\Bigg]\left|\xi_n^H\right|^{2a\delta^{n+1}}\\
\le&C\left[\left(\mathbb{E}\left|\hat{X}^\varepsilon_{n}\right|^{2a\delta^{n}}\right)^\delta+o(1)\right]\times\left[1+\left(\mathbb{E}|\xi_n^H|^{\frac{2a\delta^{n+1}}{1-\delta}}\right)^{1-\delta}\right].
\end{align*}
Since $\mathbb{E}|\hat{X}^\varepsilon_0|^{2a}=0$, by induction, we have 
\begin{align*}
\lim_{\varepsilon\to 0}\mathbb{E}\left|\hat{X}^\varepsilon_{n}\right|^{2a\delta^{n}}=0,\quad \forall n=0,1,...,N,
\end{align*}
which complete the proof of (\ref{3.4}).
\end{proof}
Then we give the stochastic maximum principle for control system (\ref{state}), (\ref{cost}).
\begin{theorem}
Let assumptions (3.1)-(3.3) hold, $u^*, X^*$ be the optimal control and the corresponding state process. Let $(p,q)$ be the solution to the following BS$\Delta$E:
\begin{align}\label{adjoint}
\left\{\begin{array}{ll}
p_n+q_n\eta_n=p_{n+1}+b_x^*(n+1)p_{n+1}+b(n+1,n+1)\sigma_x^*(n+1)q_{n+1}\\
\qquad\qquad\qquad+l_x^*(n+1)+\sigma_x^*(n+1)p_{n+1}\xi_{n+1}^H,
\\p_N=\Phi_x(X^*_N),\\
q_N=0.
\end{array}\right.
\end{align}
Then the following inequality holds:
\begin{align*}
\left[b_u^*(n)p_n+\sigma_u^*(n)p_n\sum_{k=0}^{n-1}c(n,k)\xi_k^H+b(n,n)\sigma_u^*(n)q_n+l_u^*(n)\right]\cdot(u_n-u^*_n)\ge 0,
\end{align*}
a.s., for all $n\in[0,N-1], u\in\mathbb{U}$. Here $c(n,k)=\sum_{l=0}^{n-1}b(n,l)a(l,k)$ and $a(\cdot,\cdot), b(\cdot,\cdot)$ are given by lemma \ref{lem1}.
\end{theorem}
\begin{proof}
According to the assumptions (H3.1)-(H3.2) and remark \ref{rem2}, it is easy to check BS$\Delta$E (\ref{adjoint}) satisfies  (H2.1)-(H2.3), so that BS$\Delta$E (\ref{adjoint}) has a unique solution.

Trough lemma \ref{lem3}, it is easy to show the directional derivative of $J$ is 
\begin{align}\label{3.6}
\frac{d}{d\varepsilon}J(u^*+\varepsilon v)\Big|_{\varepsilon=0}=\mathbb{E}\left[\sum_{n=0}^{N-1}\left[l_x^*(n)V_n+l_u^*(n)v_n\right]+\Phi_x(X^*_N)V_N\right].
\end{align}
Consider 
\begin{align}\label{3.7}
\Delta(p_nV_n)=&p_{n+1}V_{n+1}-p_nV_n\notag\\
=&-V_{n+1}\big[b_x^*(n+1)p_{n+1}+b(n+1,n+1)\sigma_x^*(n+1)q_{n+1}\notag\\
&\qquad\qquad+l_x^*(n+1)+\sigma_x^*(n+1)p_{n+1}\xi_{n+1}^H\big]\notag\\
&+V_{n+1}q_n\eta_n+p_n\left[b_x^*(n)V_n+b_u^*(n)v_n\right]\notag\\
&+p_n\left[\sigma_x^*(n)V_n+\sigma_u^*(n)v_n\right]\xi_n^H\notag\\
=&-\left[b_x^*(n+1)p_{n+1}V_{n+1}-b_x^*(n)p_nV_n\right]\notag\\
&-\left[\sigma_x^*(n+1)p_{n+1}V_{n+1}\xi_{n+1}^H-\sigma_x^*(n)p_nV_n\xi_n^H\right]\notag\\
&-\left[b(n+1,n+1)\sigma_x^*(n+1)q_{n+1}V_{n+1}-\sigma_x^*(n)q_nV_n\eta_n\xi_n^H\right]\notag\\
&+\left[V_n+b^*(n)V_n\right]q_n\eta_n+\sigma_u^*(n)q_nv_n\eta_n\xi_n^H\notag\\
&-l_x^*(n+1)V_{n+1}+b_u^*(n)p_nv_n+\sigma_u^*(n)p_nv_n\xi_{n}^H.
\end{align}
Notice that 
\begin{align*}
\mathbb{E}\left(\left[V_n+b^*(n)V_n\right]q_n\eta_n\right)=\mathbb{E}\left(\left[V_n+b^*(n)V_n\right]q_n\mathbb{E}\left[\eta_n|\mathcal{F}_n\right]\right)=0,
\end{align*}
and
\begin{align}\label{3.8}
\mathbb{E}\left[\sigma_x^*(n)q_nV_n\eta_n\xi_n^H\right]=&\mathbb{E}\left[\sigma_x^*(n)q_nV_n\mathbb{E}\left(\eta_n\xi_n^H|\mathcal{F}_n\right)\right]\notag\\
=&\mathbb{E}\left[\sigma_x^*(n)q_nV_n\mathbb{E}\left(\eta_n\sum_{k=0}^nb(n,k)\eta_k|\mathcal{F}_n\right)\right]\notag\\
=&\mathbb{E}\left[\sigma_x^*(n)q_nV_n\sum_{k=0}^{n-1}b(n,k)\eta_k\mathbb{E}\left(\eta_n|\mathcal{F}_n\right)\right]\notag\\
&+\mathbb{E}\left[\sigma_x^*(n)q_nV_n\mathbb{E}\left(b(n,n)\eta_n^2|\mathcal{F}_n\right)\right]\notag\\
=&\mathbb{E}\left[b(n,n)\sigma_x^*(n)q_nV_n\right],
\end{align}
so take summation and expectation of equation (\ref{3.7}), we have
\begin{align}\label{3.9}
\mathbb{E}\left[g_x(X_N^*)V_N\right]=&\sum_{n=0}^{N-1}\Delta(p_nV_n)\notag\\
=&-\mathbb{E}[b_x^*(N)p_NV_N-b_x^*(0)p_0V_0]\notag\\
&-\mathbb{E}[\sigma_x^*(N)p_NV_N\xi_N^H-\sigma_x^*(0)p_0V_0\xi_0^H]\notag\\
&-\mathbb{E}[b(N,N)\sigma_x^*(N)q_NV_N-\sigma_x^*(0)q_0V_0\xi_0^H]\notag\\
&-\mathbb{E}\sum_{n=1}^Nl_x^*(n)V_{n}+\mathbb{E}\sum_{n=0}^{N-1}b_u^*(n)p_nv_n\notag\notag\\
&+\mathbb{E}\sum_{n=0}^{N-1}[\sigma_u^*(n)p_nv_n\xi_{n}^H]+\mathbb{E}\sum_{n=0}^{N-1}[\sigma_u^*(n)q_nv_n\eta_n\xi_n^H]\notag\\
=&-\mathbb{E}\sum_{n=0}^{N-1}l_x^*(n)V_{n}+\mathbb{E}\sum_{n=0}^{N-1}b_u^*(n)p_nv_n\notag\\
&+\mathbb{E}\sum_{n=0}^{N-1}[\sigma_u^*(n)p_nv_n\xi_{n}^H]+\mathbb{E}\sum_{n=0}^{N-1}[\sigma_u^*(n)q_nv_n\eta_n\xi_n^H],
\end{align}
since $b_x^*(N)=\sigma_x^*(N)=l_x^*(N)=V_0=0$.
\end{proof}
Substitute equation (\ref{3.9}) to (\ref{3.6}), it follows
\begin{align}\label{3.10}
&\frac{d}{d\varepsilon}J(u^*+\varepsilon v)\Big|_{\varepsilon=0}\notag\\
&=\mathbb{E}\sum_{n=0}^{N-1}\left[b_u^*(n)p_n+\sigma_u^*(n)p_n\xi_n^H+\sigma_u^*(n)q_n\eta_n\xi_n^H+l_u^*(n)\right]\cdot v_n
\end{align}
Notice that, similar to equation (\ref{3.8}),
\begin{align*}
\mathbb{E}\left[\sigma_u^*(n)q_n\eta_n\xi_n^H\right]=\mathbb{E}\left[b(n,n)\sigma_u^*(n)q_n\right],
\end{align*}
and
\begin{align*}
\mathbb{E}\left[\sigma_u^*(n)p_n\xi_n^H\right]=&\mathbb{E}\left[\sigma_u^*(n)p_n\mathbb{E}\left(\sum_{l=0}^nb(n,l)\eta_l^H|\mathcal{F}_n\right)\right]\\
=&\mathbb{E}\left[\sigma_u^*(n)p_n\mathbb{E}\left(\sum_{l=0}^{n-1}\sum_{k=0}^lb(n,l)a(l,k)\xi_k^H|\mathcal{F}_n\right)\right]\\
&+\mathbb{E}\left[\sigma_u^*(n)p_nb(n,n)\mathbb{E}\left(\eta_n|\mathcal{F}_n\right)\right]\\
:=&\mathbb{E}\left[\sigma_u^*(n)p_n\sum_{k=0}^{n-1}c(n,k)\xi_k^H\right],
\end{align*}
through equation (\ref{3.10}) and the fact
$\frac{d}{d\varepsilon}J(u^*+\varepsilon v)\Big|_{\varepsilon=0}\ge 0$, we conclude
\begin{align*}
\mathbb{E}\sum_{n=0}^{N-1}\left[b_u^*(n)p_n+\sigma_u^*(n)p_n\sum_{k=0}^{n-1}c(n,k)\xi_k^H+b(n,n)\sigma_u^*(n)q_n+l_u^*(n)\right]\cdot v_n\ge 0.
\end{align*}
By the arbitrary of $v$, we have 
\begin{align*}
\mathbb{E}\left\{\mathbf{1}_\mathcal{A}\left[b_u^*(n)p_n+\sigma_u^*(n)p_n\sum_{k=0}^{n-1}c(n,k)\xi_k^H+b(n,n)\sigma_u^*(n)q_n+l_u^*(n)\right]\cdot v_n\right\}\ge 0,
\end{align*}
for all $n\in[0,N-1],\,\mathcal{A}\in\mathcal{F}_n$, which implies 
\begin{align*}
\left[b_u^*(n)p_n+\sigma_u^*(n)p_n\sum_{k=0}^{n-1}c(n,k)\xi_k^H+b(n,n)\sigma_u^*(n)q_n+l_u^*(n)\right]\cdot(u_n-u^*_n)\ge 0.
\end{align*}

\begin{remark}
If the optimal control process $(u_n^*)_{0\le n\le N-1}$ takes values in the interior of the $\mathbb{U}$, it implies
\begin{align*}
b_u^*(n)p_n+\sigma_u^*(n)p_n\sum_{k=0}^{n-1}c(n,k)\xi_k^H+b(n,n)\sigma_u^*(n)q_n+l_u^*(n)=0,\quad a.s.
\end{align*}
for all $n\in[0,N-1]$. Thus, we obtain the optimal system:
 \begin{align}\label{sys}
    \left\{\begin{array}{ll}
X_{n+1}=X_n+b(n,X_n,u_n)+\sigma(n,X_n,u_n)\xi_n^H,\\\\
p_n+q_n\eta_n=p_{n+1}+b_x^*(n+1)p_{n+1}+b(n+1,n+1)\sigma_x^*(n+1)q_{n+1}\\\\
\qquad\qquad\qquad+l_x^*(n+1)+\sigma_x^*(n+1)p_{n+1}\xi_{n+1}^H,\\\\
X_0=x,
\\\\p_N=\Phi_x(X^*_N),\\\\
q_N=0,\\\\
b_x^*(N)=\sigma_x^*(N)=l_x^*(N)=0,\\\\
b_u^*(n)p_n+\sigma_u^*(n)p_n\sum_{k=0}^{n-1}c(n,k)\xi_k^H+b(n,n)\sigma_u^*(n)q_n+l_u^*(n)=0.
\end{array}\right.
\end{align}

\begin{corollary}
If $\xi^H$ is white noise ($H=1/2$), then it is easy to check
\begin{equation}
a(n,k)=b(n,k)=\left\{
\begin{aligned}
&1,\ \ {\rm if}\ n=k,
\\
&0,\ \ {\rm if}\ n\neq k,
\end{aligned}
\right.
\notag
\end{equation}
so that $c(n,k)=0, \,k=0,1,...,n-1$. Moreover, consider the adjoint BS$\Delta$E, $p$ and $q$ are given by
\begin{align*}
p_n=\mathbb{E}\Big[p_{n+1}+&b_x^*(n+1)p_{n+1}+b(n+1,n+1)\sigma_x^*(n+1)q_{n+1}\\
&+l_x^*(n+1)+\sigma_x^*(n+1)p_{n+1}\xi_{n+1}^H|\mathcal{F}_n\Big]\\
=\mathbb{E}\Big[p_{n+1}+&b_x^*(n+1)p_{n+1}+\sigma_x^*(n+1)q_{n+1}+l_x^*(n+1)|\mathcal{F}_n\Big],
\end{align*}
and
\begin{align*}
q_n=\mathbb{E}\Big[\eta_n\big[p_{n+1}+&b_x^*(n+1)p_{n+1}+\sigma_x^*(n+1)q_{n+1}+l_x^*(n+1)\big]|\mathcal{F}_n\Big].
\end{align*}
We can also write the adjoint equation $(p,q)$ as 
\begin{align*}
\left\{\begin{array}{ll}
p_n+q_n\eta_n=p_{n+1}+b_x^*(n+1)p_{n+1}+\sigma_x^*(n+1)q_{n+1}+l_x^*(n+1),
\\p_N=\Phi_x(X^*_N),\\
q_N=0.
\end{array}\right.
\end{align*}
The condition for optimal control is 
\begin{align*}
b_u^*(n)p_n+\sigma_u^*(n)q_n+l_u^*(n)=0, \quad a.s.\quad n=0,1,...,N-1,
\end{align*}
which is same to the results obtained in \cite{lin2014maximum}
\end{corollary}

\end{remark}

\section{Application to linear quadratic control}
In this section, we consider the following linear quadratic (LQ) optimal control problem, the state equation is 
\begin{align}\label{4.1}
\left\{\begin{array}{ll}
X_{n+1}=X_n+A_nX_n+B_nu_n+\left[C_nX_n+D_nu_n\right]\xi_n^H,\, 0\le n\le N-1,
\\X_0=x,
\end{array}\right.
\end{align}
with the cost function 
\begin{align}\label{4.2}
J(u)=\frac{1}{2}\mathbb{E}\left[\sum_{n=0}^{N-1}\left(Q_nX_n^2+R_nu^2_n\right)+GX_N^2\right],
\end{align}
where $Q_n,G\ge 0$ and $R_n>0$ for $n=0,1,...,N-1$.

According to the optimal system (\ref{sys}), the adjoint equation is
\begin{align}\label{4.3}
\left\{\begin{array}{ll}
p_n+q_n\eta_n=p_{n+1}+A_{n+1}p_{n+1}+b(n+1,n+1)C_{n+1}q_{n+1}\\
\qquad\qquad\qquad+Q_{n+1}X^*_{n+1}+C_{n+1}p_{n+1}\xi_{n+1}^H,
\\p_N=GX^*_N,\\
q_N=0,\\
A_N=C_N=Q_N=0.
\end{array}\right.
\end{align}
The optimal control should satisfy
\begin{align*}
B_np_n+D_np_n\sum_{k=0}^{n-1}c(n,k)\xi_k^H+b(n,n)D_nq_n+R_nu_n^*=0,
\end{align*}
i.e.,
\begin{align}\label{4.4}
u_n^*=-R_n^{-1}\left[B_np_n+D_np_n\sum_{k=0}^{n-1}c(n,k)\xi_k^H+b(n,n)D_nq_n\right].
\end{align}
\begin{theorem}
Let $(p,q)$ be the solution of BS$\Delta$E (\ref{4.3}). Then function (\ref{4.4}) is the unique optimal control for control problem (\ref{4.1}), (\ref{4.2}).
\end{theorem}
\begin{proof}
Let us show the sufficiency of the optimal control first. Assume that $u\in\mathbb{U}$ is any other admissible control, $\{X_n\}$ is the corresponding state process, denote $\Delta X_n=X_n-X_n^*$ and $\Delta u_n=u_n-u^*_n$, then
\begin{align}\label{4.1}
\left\{\begin{array}{ll}
\Delta X_{n+1}=\Delta X_n+A_n\Delta X_n+B_n\Delta u_n+[C_n\Delta X_n+D_n\Delta u_n]\xi_n^H
\\\Delta X_0=0,
\end{array}\right.
\end{align}
and 
\begin{align*}
&p_{n+1}\Delta X_{n+1}-p_n\Delta X_n\\
=&-\Delta X_{n+1}\big[A_{n+1}p_{n+1}+b(n+1,n+1)C_{n+1}q_{n+1}\notag\\
&\qquad\qquad+Q_{n+1}X^*_{n+1}+C_{n+1}p_{n+1}\xi_{n+1}^H\big]\notag\\
&+\Delta X_{n+1}q_n\eta_n+p_n\left[A_n\Delta X_n+B_n\Delta u_n\right]\notag\\
&+p_n\left[C_n\Delta X_n+D_n\Delta u_n\right]\xi_n^H\notag\\
=&-\left[A_{n+1}p_{n+1}\Delta X_{n+1}-A_np_n\Delta X_n\right]\notag\\
&-\left[C_{n+1}p_{n+1}\Delta X_{n+1}\xi_{n+1}^H-C_np_n\Delta X_n\xi_n^H\right]\notag\\
&-\left[b(n+1,n+1)C_{n+1}q_{n+1}\Delta X_{n+1}-C_nq_n\Delta X_n\eta_n\xi_n^H\right]\notag\\
&-Q_{n+1}X^*_{n+1}\Delta X_{n+1}+B_np_n\Delta u_n+D_np_n\Delta u_n\xi_{n}^H\notag\\
&+D_nq_n\Delta u_n\eta_n\xi_n^H.
\end{align*}
So that
\begin{align*}
&\mathbb{E}\left(GX^*_N\Delta X_N\right)\\=&\sum_{n=0}^{N-1}\mathbb{E}\Big[-Q_{n+1}X^*_{n+1}\Delta X_{n+1}+B_np_n\Delta u_n+D_np_n\Delta u_n\xi_{n}^H
+D_nq_n\Delta u_n\eta_n\xi_n^H\Big]\\
=&\sum_{n=0}^{N-1}\mathbb{E}\Big[-Q_{n+1}X^*_{n+1}\Delta X_{n+1}+B_np_n\Delta u_n\\
&\qquad\qquad+D_np_n\Delta u_n\sum_{k=0}^{n-1}c(n,k)\xi_k^H
+b(n,n)D_nq_n\Delta u_n\eta_n\Big]
\\=&\sum_{n=0}^{N-1}\mathbb{E}\Big[-Q_{n}X^*_{n}\Delta X_{n}-R_nu^*_n\Delta u_n\Big].
\end{align*}
Then
\begin{align*}
J(u)-J(u^*)=&\frac{1}{2}\mathbb{E}\left[\sum_{n=0}^{N-1}\left(Q_n(X_n^2-X_n^{*2})+R_n(u^2_n-u_n^{*2})\right)+G\left(X_N^2-X_N^{*2}\right)\right]\\
\ge &\mathbb{E}\left[\sum_{n=0}^{N-1}\left(Q_nX_n^*\Delta X_n+R_nu^*_n\Delta u_n\right)+GX_N^*\Delta X_N\right]\\
=&0,
\end{align*}
which shows $u^*$ is an optimal control. 

Then we show uniqueness of the optimal control. Let  both $u_t^{*,1}$ and $u_t^{*,2}$ are optimal control processes, $X_t^1$ and $X_t^2$ are corresponding state processes, respectively. It is easy to check $\frac{X_t^1+X_t^2}{2}$ is the corresponding state process to $\frac{u_t^{*,1}+u_t^{*,2}}{2}$. Assume that there exists constants $\theta>0, \alpha\ge 0$, such that $R_t\ge \theta$ and 
\begin{align*}
    J(u_t^{*,1})=J(u_t^{*,2})=\alpha.
\end{align*}
Using $a^2+b^2=2\left[(\frac{a+b}{2})^2+(\frac{a-b}{2})^2\right]$, we have that

\begin{align*}
2\alpha=&J(u^{*,1})+J(u^{*,2})\\
=&\frac{1}{2}\mathbb{E}\sum_{n=0}^{N-1} \Big[Q_n(X_n^1X_n^1+X_n^2X_n^2)+R_n(u_n^{*,1}u_n^{*,1}+u_n^{*,2}u_n^{*,2})\Big]\\
&+\frac{1}{2}\mathbb{E}G(X_N^1X_N^1+X_N^2X_N^2)\\
\ge&\mathbb{E}\sum_{n=0}^{N-1} \left[Q_n\Big(\frac{X_n^1+X_n^2}{2}\Big)^2+R_n\Big(\frac{u_n^{*,1}+u_n^{*,2}}{2}\Big)^2\right]\\
&+\mathbb{E}G\Big(\frac{X_N^1+X_N^2}{2}\Big)^2+\mathbb{E}\sum_{n=0}^{N-1}R_n\Big(\frac{u_n^{*,1}-u_n^{*,2}}{2}\Big)^2\\
=&2J\Big(\frac{u^{*,1}+u^{*,2}}{2}\Big)+\mathbb{E}\sum_{n=0}^{N-1}R_n\Big(\frac{u_n^{*,1}-u_n^{*,2}}{2}\Big)^2\\
\ge&2\alpha+\frac{\theta}{4}\mathbb{E}\sum_{n=0}^{N-1}|u_t^{*,1}-u_t^{*,2}|^2.
\end{align*}
Thus, we have 
\begin{align*}
\mathbb{E}\sum_{n=0}^{N-1}|u_t^{*,1}-u_t^{*,2}|^2=0,
\end{align*}
which shows the uniqueness of the optimal control.

\end{proof}

\section{Acknowledgments}

This work was supported by National Key R$\&$D Program of China (Grant number 2023YFA1009200) and National Science Foundation of China (Grant numbers 12471417).

\FloatBarrier
\bibliography{main}
\end{document}